\documentclass[11pt]{article}
\usepackage{pdfsync}
\usepackage{hyperref}
\usepackage{array}
\usepackage{amsfonts}
\usepackage{amsmath}
\usepackage{amssymb}
\usepackage{amsthm}

\usepackage{graphicx}
\usepackage{color}

\newtheorem{thm}{Theorem}[section]
\newtheorem{cor}[thm]{Corollary}
\newtheorem{lemma}[thm]{Lemma}
\newtheorem{prop}[thm]{Proposition}

\theoremstyle{remark}
\newtheorem{remark}[thm]{Remark}
\newtheorem*{remark*}{Remark}

\numberwithin{equation}{section}

\newcommand\EE{{\mathcal{E}}}
\newcommand\FF{{\mathcal{F}}}
\newcommand\JJ{{\mathcal{J}}}
\newcommand\UU{{\mathcal{U}}}
\newcommand\UUseg{{\mathcal{U}_{\text{seg}}}}
\newcommand\eps{{\varepsilon}}
\newcommand\R{{\mathbb R}}

\newcommand\uu{{\mathbf u}}
\newcommand\vv{{\mathbf v}}

\newcommand\ww{{\mathbf w}}
\newcommand\ebeta{_{\eps,\beta}}
\newcommand\dx{{\,dx}}
\newcommand\dt{{\,dt}}

\newcommand\emt{{e^{-t/\eps}}}
\newcommand\de{\mathop{\partial_t\!}}

\title{A minimization procedure to the existence of segregated \\ solutions to parabolic reaction-diffusion systems
\footnote{Work partially supported by the ERC Advanced Grant No. 339958 ``Complex Patterns for Strongly Interacting Dynamical Systems - COMPAT'' and by the ERC Starting Grant No. 801867 ``Regularity and singularities in elliptic PDE - EllipticPDE''.}}

\author{Alessandro Audrito$^{1,2}$,  Enrico Serra$^2$,
Paolo Tilli$^2$ \\ \ \\{\small  $^1$ Institut f\"ur Mathematik,
Universit\"at Z\"urich}\\
{\small
Winterthurerstrasse 190
CH-8057 Z\"urich}\\
{\small $^2$ Dipartimento di Scienze
Matematiche ``G.L. Lagrange'', Politecnico di Torino } \\ {\small
Corso Duca degli Abruzzi, 24, 10129 Torino, Italy}}

\date{}

\begin{document}

\maketitle

\begin{abstract}
We study the existence of segregated solutions to a class of reaction-diffusion systems with strong interactions, arising in many physical applications. These special solutions are obtained as weak limits of minimizers of a family of perturbed functionals. We prove some a priori estimates through a minimization procedure which is nonstandard in the parabolic theory: our approach is purely variational and all the information is encoded in the family of functionals we consider.
\end{abstract}



\section{Introduction}\label{SectionIntroduction}

Let $\Omega \subset \mathbb{R}^N$ be a bounded domain with smooth boundary and
let  $A := (a_{ij})_{i,j}$ be a $k \times k$ symmetric matrix ($k \ge 2$)
satisfying  $a_{ii} = 0$  and $a_{ij} > 0$ if $i \not= j$. We study the existence of a special class of solutions to the system of parabolic equations

\begin{equation}
\label{parsys}
\de v_i - \Delta v_i = f_i(v_i) - \beta \, v_i \sum_{j \not = i} a_{ij} v_j^2 \quad \text{ in } \Omega \times (0,+\infty),\qquad i= 1,\dots, k
\end{equation}
in the \emph{strong interaction} regime $\beta \to + \infty$. To fix ideas, one should think that the nonlinearities $f_i$ behave like the logistic reaction $f_i(s) = s(1-s)$.

Each equation is coupled with some boundary conditions assigned on the parabolic boundary $\Omega \times\{0\} \cup \partial \Omega \times \R^+ $: for example, we can assume

\begin{equation}
\label{bc}
\begin{cases}
v_i(x,t) = g_i(x),         & (x,t) \in \partial\Omega \times (0,+\infty) \\
v_i(x,0) = v_{0i}(x),    & x \in \Omega,
\end{cases}
\end{equation}
for a suitable choice of functions $g_i$ and $v_{0i}$, $i = 1,\ldots,k$, but also other kinds of boundary conditions can be considered. In any case, we always suppose that the the initial data $v_{0i}$ satisfy the {\em segregation condition}
\[
v_{0i} \cdot v_{0j} = 0 \quad \text{a.e. in } \Omega \quad \text{ for all } i \not = j.
\]
This class of systems appears in a wide variety of biological and physical applications, such as population dynamics and Bose-Einstein condensates. Notice that the time variation of each component $v_i$ is given by the combination of three terms: the first is diffusion, the second is given by the $i$-th internal energy $f_i$, while the third  is  the sum of \emph{penalizing interaction} terms, which become stronger and stronger as $\beta$ grows: investigating the singular range $\beta = +\infty$ is the main goal of this paper.

In  \cite{DanWanZha2012_1:art}, Dancer, Wang and Zhang proved that if a sequence of bounded  solutions to \eqref{parsys}--\eqref{bc} converges as $\beta \to +\infty$ to some $(w_1,\dots,w_n)$, then these functions are weak solutions to the system of differential inequalities
\begin{equation}
\label{eq:DiffIneq}
\begin{cases}
(\de - \Delta ) w_i  \leq f_i(w_i),\quad  i = 1,\ldots,k  \\ \\
(\de - \Delta) \Big( w_i - \sum\limits_{j\not=i}w_j \Big ) \geq f_i(w_i) - \sum\limits_{j\not=i} f_j(w_j),
\quad i = 1,\ldots,k,
\end{cases}
\end{equation}
that have the \emph{segregated} form
\[
w_i \cdot w_j = 0 \quad \text{a.e. in } \Omega \times (0,+\infty) \quad \text{ for all } i \not = j
\]
and satisfy the same boundary conditions \eqref{bc}.
Such solutions are interesting both from  the mathematical and the physical viewpoint and, in the present work, will be obtained through a nonstandard procedure, based on the minimization of a special functional introduced by  De Giorgi in the context of nonlinear wave equations. Before describing our approach, a quick bibliographical survey is  in order.

The study of segregated solutions has been quite popular in past and recent years. The first works can be dated back to the 90's (see for instance Dancer \cite{Dan1991:art}, and Dancer and Du \cite{DanDu1994:art}). More recently, much work has been carried out in the \emph{elliptic} framework, namely the study of  Gross-Pitaevskii-type systems

\begin{equation}\label{eq:EllipticSystemIntro}
- \Delta v_i = f_i(v_i) - \beta \, v_i \sum_{j \not = i} a_{ij} v_j^2 \quad \text{ in } \Omega
\end{equation}
for $i = 1,\ldots,k$ and $\beta \to +\infty$. The existence of segregated solutions was studied
by Conti, Terracini and Verzini in \cite{ConTerVer2002:art,ConTerVer2003:art,ConTerVer2005:art}, where the authors established a connection between the singular limit of minimal solutions to \eqref{eq:EllipticSystemIntro} and  the solutions of a related optimal partition problem. Independently, similar results were obtained by Caffarelli and Lin \cite{CafLin2007:art,CafLin2008:art} (see also \cite{CafSal2005:book}). Subsequently,
in a series of articles \cite{ConTerVer2005_1:art,NorTavTerVer2010:art,SoaveZilio2015:art,TavTer2012:art}, it was shown that solutions to \eqref{eq:EllipticSystemIntro} are locally uniformly H\"older (or even Lipschitz) continuous and finer properties of the limiting solutions and their \emph{free boundaries} were established.

The parabolic setting is less studied and seems to be more delicate. We quote the works of Dancer, Wang and Zhang \cite{DanWanZha2011:art,DanWanZha2012_1:art}, where the authors proved uniform H\"{o}lder (and Lipschitz) estimates of parabolic type and derived system \eqref{eq:DiffIneq} for general solutions to \eqref{parsys}, under some convergence assumptions, when $\beta \to +\infty$ (see also \cite{DanWanZha2012:art,WangZhang2010:art} for the Lotka-Volterra framework). The proof is quite involved and uses several  advanced technical tools such as Almgren and Alt-Caffarelli-Friedman monotonicity formulas, as well as blow-up and dimension reduction arguments.  We finally quote the works of Caffarelli and Lin, \cite{CafLin2009:art}, and Snelson, \cite{Snelson2015:art}, for a parallel study in the context of gradient flows with singular values.

The aim of this work is twofold. First, if compared to \cite{DanWanZha2011:art,DanWanZha2012:art,DanWanZha2012_1:art,WangZhang2010:art}, our approach is somehow inverse: in these papers the authors study the asymptotic behaviour as $\beta \to +\infty$ of general solutions to \eqref{parsys}--\eqref{bc}, while we approximate \eqref{bc}--\eqref{eq:DiffIneq} through a new family of perturbed problems. Our method can be seen as the \emph{parabolic} counterpart of the study carried out in \cite{ConTerVer2005:art} in the \emph{elliptic} framework, where difficulties can be overcome since \eqref{eq:EllipticSystemIntro} has a clear variational structure while, in the parabolic setting, this structure is apparently lacking. As a consequence, we obtain a very direct proof of the existence of \emph{bounded global segregated} weak solutions to problem \eqref{bc}--\eqref{eq:DiffIneq}.

Secondly, as we have mentioned above, we want to draw the attention on the fact that this class of solutions is obtained through a {\em minimization procedure}, based on a conjecture of De Giorgi in a different setting (\cite{DeGiorgi1996:art}), which is quite innovative in this framework. It is a very direct and flexible method: we recover the  convergence results of the elliptic setting, \cite{ConTerVer2002:art,ConTerVer2005:art,NorTavTerVer2010:art,TavTer2012:art} and, further, we will extend it to more general systems of parabolic equations in a forthcoming paper.

We now describe, in a rather informal way, the main ideas involved in our approach, postponing precise definitions and statements to subsequent sections.

For  $\eps, \beta > 0$ we consider the functional

\begin{equation}\label{eq:DeGiorgiFunctional}
\FF_{\eps,\beta}(\vv) := \int_0^\infty\!\!\int_{\Omega} \frac{e^{-t/\eps}}{\eps}  \Big\{ \varepsilon \sum_{i=1}^k |\de v_i |^2  + \sum_{i=1}^k \left[|\nabla v_i|^2 - 2F_i(v_i)\right] + \frac{\beta}{2}\sum_{i,j = 1}^ka_{ij} v_i^2v_j^2 \Big\} \,dx dt,
\end{equation}
where $F_i(s) = \int_0^s f_i(s)ds$ for each $i$. If the functional $\FF_{\eps,\beta}$ has a global minimizer  $\vv_{\eps,\beta} = (v_1,\dots,v_k)$ in some suitable space (and among functions satisfying \eqref{bc}), then for every test function $\eta\in C^\infty_0(\Omega\times \R^+)$
the Euler--Lagrange equations read
\[
\int_0^\infty\!\!\int_{\Omega} \frac{e^{-t/\eps}}{\eps} \Big\{\eps \de v_i \de\eta + \nabla v \cdot\nabla\eta - F_i'(v_i)\eta + \beta v_i \sum_{j \not = i} a_{ij} v_j^2\eta  \Big\} \,dxdt =0,\quad i=1,\ldots,k,
\]
which is the weak form of

\begin{equation}
\label{strong}
- \eps \partial_{tt} v_i + \de v_i -  \Delta v_i=   F'_i(v_i) - \beta \, v_i \sum_{j\not=i} a_{ij} v_j^2 = 0,\quad i=1,\ldots,k.
\end{equation}
Since $f_i=F_i'$,  this is system \eqref{parsys}, with the addition of the extra term $\eps \partial_{tt} v_i$. Assuming now that one has suitable boundedness and compactness properties (with respect to the parameters $\eps$ and $\beta$), one can try to pass to the limit as $\eps \to 0$ and $\beta \to +\infty$, in order to obtain \emph{segregated} solutions to \eqref{bc}--\eqref{eq:DiffIneq}. There are basically two options.

The first one is to keep $\beta > 0$ fixed and take the limit as $\eps \to 0$: in this case one could prove that the term $\eps \partial_{tt} v_i$ goes to zero and obtain limit solutions that satisfy \eqref{parsys}--\eqref{bc} and so, as a consequence of \cite{DanWanZha2011:art,DanWanZha2012_1:art}, construct segregated weak solutions to \eqref{bc}--\eqref{eq:DiffIneq} passing to the limit as $\beta \to +\infty$. As we have mentioned above, studying the asymptotic regime $\beta = +\infty$ is however a difficult problem and involves several technical tools.

Conversely, the option we follow is to fix $\eps \in (0,1)$ and take the limit as $\beta \to +\infty$. In this way, the coefficient of the last term in the functional $\FF_{\eps,\beta}$ goes to infinity  (strong interaction regime) and suggests that

\[
v_i \cdot v_j \to 0 \quad \text{ as } \beta \to +\infty, \quad i \not= j,
\]
as expected.  Indeed, for any fixed $\eps \in (0,1)$, the family of minimizers $\{\vv_{\eps,\beta}\}_{\eps,\beta}$ of $\FF_{\eps,\beta}(\vv)$ converges as $\beta \to \infty$ (in a suitable sense) to a {\em global minimizer} $\vv_\eps$ of

\begin{equation}\label{defFeps}
\FF_{\varepsilon}(\vv) :=    \int_0^\infty\!\!\int_{\Omega} \frac{e^{-t/\eps}}{\eps}  \Big\{ \varepsilon \sum_{i=1}^k |\de v_i |^2  + \sum_{i=1}^k \left[|\nabla v_i|^2 - 2F_i(v_i)\right] \Big\} \,dx dt
\end{equation}
having \emph{segregated} components $v_i \cdot v_j = 0$, $i \not= j$. Notice that the corresponding Euler-Lagrange equations are now \eqref{strong} without the interaction terms (i.e. $\beta = 0$). Then we study the limit of $\vv_\eps$ as $\eps \to 0$ and we prove that the limit functions (that have segregated components for free) solve \eqref{eq:DiffIneq}.

To carry out this program it is therefore necessary to obtain uniform bounds in $\eps$ and $\beta$ for the sequence $\vv = \vv_{\eps,\beta}$, and this will be the central part of our argument. We point out however that there is no need to consider equation \eqref{strong}, as all the information is embodied in the functional $\FF_{\eps,\beta}$, and we will accordingly work directly with $\FF_{\eps,\beta}$. 

We believe that this is one of the main points of interest of our work, and constitutes a major novelty in the approach to parabolic segregation problems.
Note also that each functional $\FF_{\eps,\beta}$ is the energy functional associated to an {\em elliptic} problem, once one neglects the interpretation of the variable $t$ as time.  Working with $\FF_{\eps,\beta}$ may be considerably easier than working directly with the parabolic system
\eqref{parsys}, since one can make use of all the machinery proper of elliptic problems, such as the techniques of the Calculus of Variations. This could be of great help if one is interested in proving further regularity and finer properties of the solutions.

On the other hand, as $\eps \to 0$, two problems emerge. The first is that ellipticity is not uniform (and it has to be so, since in the limit we must recover a parabolic problem). The second is due to the presence of the exponential weight in the functional, which degenerates very rapidly as $\eps \to 0$. These two features  make it apparently very  hard to obtain the {\em uniform} estimates in $\eps$ which are needed to complete the limit procedure.

However, a way to overcome this difficulties has been introduced, for hyperbolic problems, in \cite{SerraTilli2012:art,SerraTilli2016:art} (see also \cite{TentTilli2018:art,TentTilli2018_1:art}),
and we show here how to  profit from it  in the parabolic  setting.

This approach is also known in the literature as the\emph{Weighted Inertia-Energy-Dissipation (WIDE)} method, appeared first in the works of Lions and Oleinik (see e.g. \cite{Lions1965:art,LionsMagenes1968:art,Oleinik1964:art}). Successively, it has been investigated by many authors: we quote the works of Akagi, Mielke, Ortiz and Stefanelli \cite{AkagiStefanelli2016:art,MielkeOrtiz2008:art,MielkeStefanelli2008:art,MielkeStefanelli2011:art}, B\"ogelein et al. \cite{BogeleinEtAl2015:art,BogeleinEtAl2017:art} and the references therein. However, our techniques are inspired by the methods used in \cite{SerraTilli2012:art, SerraTilli2016:art} and differ from those of the just mentioned articles.

Summing up, the main questions we address in this paper are the following: does \eqref{eq:DeGiorgiFunctional} have a minimizer for each $\eps \in (0,1)$ and $\beta > 0$? Given a family of minimizers $\{\vv_{\eps,\beta}\}_{\eps,\beta}$ of \eqref{eq:DeGiorgiFunctional}, are there limit functions $\ww$ as $\beta \to +\infty$ and $\eps \to 0$? If yes, what system do they satisfy? Are these segregated solutions? The aim of the present work is to give an answer to these questions.

This paper is organized as follows.
In Section \ref{mainres} we describe the precise functional setting and we state the main results, while in Section \ref{SectionExistenceofMinimizers} we prove  the existence of minimizers of $\FF_{\eps,\beta}$.
Section \ref{SectionEstimates} contains the main estimates, while in Section \ref{double} we study the asymptotic behavior of the family $\{\vv_{\eps,\beta}\}_{\eps,\beta}$ as $\beta \to +\infty$ and $\eps \to 0$ and we prove our main result. Finally, in Section \ref{SectionDiscussions}, we show how analogous results known in the literature for  elliptic problems are contained as a particular case of our result.

\section{Functional setting and main results}
\label{mainres}

From now on we assume that $\eps \in (0,1)$ and $\beta > 0$. Furthermore, we assume that the $k$ functions $F_i:\mathbb{R} \to \mathbb{R}$ satisfy

\begin{equation}\label{eq:AssumptionsPotential}
\begin{cases}
F_i \text{ is continuous on } \mathbb{R} \text{ and differentiable in } [0,1) \\
F_i \text{ is non-decreasing in } (-\infty,0) \text{ and non-increasing in } (1,+\infty) \\
F_i'(0) = 0,
\end{cases}
\end{equation}
for all $i = 1,\ldots,k$. A simple example is obtained by letting $F_i$ be a primitive of  $f_i(v) = v^2(1-v)$.  Our results are valid under weaker assumptions than \eqref{eq:AssumptionsPotential} (as for instance in \cite[Section 2]{ConTerVer2005:art}) but, for simplicity, we limit ourselves to this more basic setting.

Given a vector-valued function $\vv = (v_1,\ldots,v_k)$, writing
\[
|\de \vv|^2 := \sum_{i=1}^k|\de v_i|^2, \quad |\nabla\vv|^2 := \sum_{i=1}^k |\nabla v_i|^2, \quad \langle \vv^2 ,A \vv^2\rangle := \sum_{i,j=1}^k a_{ij} v_i^2v_j^2, \quad F(\vv) := \sum_{i =1}^k F_i(v_i),
\]
we obtain for  the functional $\FF_{\eps,\beta}$ defined
in \eqref{eq:DeGiorgiFunctional} the more compact expression

\begin{equation}
\label{dgv}
\FF_{\eps,\beta}(\vv) :=  \int_0^\infty\!\!\int_{\Omega} \frac{e^{-t/\eps}}{\eps} \Big\{ \eps |\de \vv|^2 + |\nabla \vv|^2 - 2F(\vv) + \frac{\beta}{2} \langle \vv^2, A \vv^2 \rangle  \Big\} \,dxdt.
\end{equation}
A natural domain for this functional is the space
of vector-valued functions

\begin{equation}
\label{defUU}
\UU:=\bigcap_{T>0}  H^1 \left( \Omega_T \right)^k,
\qquad \Omega_T:=\Omega\times (0,T),
\end{equation}
endowed with its natural topology.
 We can view $\FF_{\eps,\beta}$ as a functional defined on $\UU$, and taking values in $(-\infty,+\infty]$.
Indeed, since $a_{ij}\geq 0$, when $\vv\in\UU$ all the
integrands in \eqref{dgv}, with the exception of
$F(\vv)$, are nonnegative, hence their integrals
are well-defined (possibly equal to $+\infty$). On the other
hand, thanks to \eqref{eq:AssumptionsPotential}, every function $F_i$ has a
finite upper bound, hence also the integral of $-2F(\vv)$
in \eqref{dgv} takes values in $(-\infty,+\infty]$.

Moreover, every function  $\vv \in \UU$ has a trace on the parabolic boundary $\Omega\times\{0\} \,\cup\, \partial\Omega\times \R^+$: in particular, the function  $x\mapsto \vv(x,0)$ (the ``initial value'') belongs to $H^{1/2}(\Omega)^k$,
while for a.e. $t>0$ the function $x\mapsto\vv(x,t)$, being an element of $H^1(\Omega)^k$,
has a trace in $H^{1/2}(\partial\Omega)^k$
(the ``boundary value'').

This allows
us to consider the minimization of the functional $\FF$ on $\UU$ subject to
suitable initial and boundary conditions, as follows. Let
\begin{equation}
\label{eq:InitialConditions}
\mathbf{v_0}=(v_{01},\ldots,v_{0k}) \in H^1(\Omega)^k
\end{equation}
be a function satisfying the bounds
\begin{equation}
\label{bound}
0 \leq \mathbf{v_0} \leq 1
\end{equation}
(i.e. $0 \leq v_{0i} \leq 1$ for $i = 1,\ldots,k$) together with the \emph{segregation condition}

\begin{equation}
\label{eq:SegregationCondition}
v_{0i}\cdot v_{0j} = 0 \quad \text{a.e. in } \Omega, \quad \text{ for every } i \not = j,
\end{equation}
and let $\mathbf{g}_0\in H^{1/2}(\partial\Omega)^k$ denote the trace of $\mathbf{v_0}$ on $\partial \Omega$.
We will minimize the functional $\FF$ on the set of functions

\begin{equation}\label{eq:DFSpace}
\UU_{\vv_0,\mathbf{g}_0} := \{\vv \in \UU: \; \vv(\cdot,0) = \vv_0,  \;\;
\vv(\cdot,t) = \mathbf{g}_0 \text{ on $\partial\Omega$ for a.e. $t>0$}\},
\end{equation}
that is, among all functions having $\mathbf{v_0}$ as initial condition at time $t=0$, and
$\mathbf{g_0}$ as Dirichlet boundary condition on $\partial\Omega$
at almost every time $t>0$ (time-dependent boundary
conditions might also be considered, but we will not pursue this here).
Other options are possible: for instance, one may drop the
boundary condition $\vv = \mathbf{g}_0$  and minimize within the larger set of functions

\begin{equation}\label{eq:NFSpace}
\UU_{\vv_0} := \{\vv \in \UU: \; \vv(\cdot,0) = \vv_0  \}
\end{equation}
or, dually, drop the initial condition and consider the set of functions

\begin{equation}\label{eq:EFSpace}
\UU_{\mathbf{g}_0} := \{\vv \in \UU: \;
\vv(\cdot,t) = \mathbf{g}_0 \text{ on $\partial\Omega$ for a.e. $t>0$}\}.
\end{equation}
We point out that the three sets $\UU_{\vv_0,\mathbf{g}_0}$, $\UU_{\vv_0}$
and $\UU_{\mathbf{g}_0}$ are convex and closed in $\UU$ and, since
\begin{equation}
  \label{nonvuoto}
  \widetilde \vv\in \UU_{\vv_0,\mathbf{g}_0}=\UU_{\vv_0}
\cap\UU_{\mathbf{g}_0},\quad \text{where $\,\widetilde\vv(x,t):=\vv_0(x),$}
\end{equation}
each of these sets is nonempty. The last spaces we need to introduce are denoted by $\mathcal{U}_{seg}$ and $\mathcal{S}$, first introduced in \cite{ConTerVer2005:art} for the elliptic problem, and then extended in \cite{DanWanZha2012_1:art,WangZhang2010:art,Wang2015:art} to the parabolic setting.

The space $\mathcal{U}_{seg}$ is the space of \emph{segregated configurations}

\[
\mathcal{U}_{seg} := \bigl\{\vv=(v_1,\ldots,v_k) \in \mathcal{U}_0 : v_i \cdot v_j = 0 \text{ a.e. in } \Omega\times\R^+, \text{ for all } i \not= j\bigr\},
\]
where $\mathcal{U}_0$ denotes either $\UU_{\vv_0,\mathbf{g}_0}$ or $\UU_{\vv_0}$, while $\mathcal{S}$ is made of all \emph{segregated} functions $\vv \in \mathcal{U}_{seg}$, satisfying the system of differential inequalities \eqref{eq:DiffIneq} in the sense of distributions in $\Omega\times\R^+$.
This means that, setting $\widehat{v}_i := v_i - \sum_{j\not=i} v_j$ and $\widehat{f}_i(v_i) := f_i(v_i) - \sum_{j\not=i} f_j(v_j)$, the inequalities

\begin{align*}
&\int_0^{\infty}\int_{\Omega} \left[ \de v_i \eta + \nabla v_i \cdot \nabla \eta - f_i(v_i) \eta \right] dxdt \leq 0, \\
&\int_0^{\infty}\int_{\Omega} \left[ \de \widehat{v}_i \eta + \nabla \widehat{v}_i \cdot \nabla \eta - \widehat{f}_i(v_i)\eta \right] dxdt \geq 0
\end{align*}
hold true for every nonnegative $\eta \in \mathcal{C}_c^{\infty}(\Omega\times\R^+)$. Summing up,

\begin{equation}
\label{eq:Sclass}
\mathcal{S} := \left\{\vv \in \mathcal{U}_{seg}:\,\,
\text{$\vv = (v_1,\ldots,v_k)$ satisfies \eqref{eq:DiffIneq}}
\right\}.
\end{equation}
The present paper is divided into three main sections in which we prove our three main theorems. The first result is quite standard and concerns the existence of global minimizers of $\FF_{\eps,\beta}$, when $\eps \in (0,1)$ and $\beta > 0$ are fixed.

\begin{thm}
\label{minimizers}
Let $\UU_0$ denote one of $\UU_{\vv_0,\mathbf{g}_0}$, $\UU_{\vv_0}$ or $\UU_{\mathbf{g}_0}$.
For every $\eps \in(0,1)$ and $\beta > 0$, the functional $\FF_{\eps,\beta}$ has an absolute minimizer in $\UU_0$. Moreover, every minimizer $\vv=(v_1,\ldots,v_k)$ satisfies $0 \le v_i \le 1$ a.e. in $\Omega \times \R^+$, for every $i=1,\dots,k$.
\end{thm}

In the second theorem, we prove some Sobolev-type estimates on minimizers $\vv_{\eps,\beta}$ of $\FF_{\eps,\beta}$, which are uniform with respect to $\eps \in (0,1)$ and $\beta > 0$. This is the central part of the work: it shows that the family of minimizers $\{\vv_{\eps,\beta}\}_{\eps,\beta}$ is precompact in the weak topology of $\mathcal{U}$.

\begin{thm}\label{ape}
Let $\UU_0$ denote either $\UU_{\vv_0,\mathbf{g}_0}$ or $\UU_{\vv_0}$. Let $\vv= \vv_{\eps,\beta}$ be a minimizer of $\FF_{\eps,\beta}$ in $\UU_0$. Then

\begin{equation}\label{uno}
\int_\tau^{\tau+T}\!\!\!\int_{\Omega} \left\{|\nabla \vv|^2
+ \beta\langle \vv^2, A \vv^2 \rangle\right\} \,dxdt \leq CT \quad\text{$\forall \tau\geq 0\quad\forall T\geq\eps,$}
\end{equation}
\begin{equation}\label{due}
\|\vv\|_{L^\infty(\Omega\times \R^+)^k} \leq C,
\end{equation}

\begin{equation}\label{tre}
\int_0^{\infty}\!\!\!\int_{\Omega} |\partial_{\tau} \vv|^2 \,dxdt \leq C,
\end{equation}
for some constant $C$ independent of $\eps \in (0,1)$ and $\beta > 0$.
\end{thm}

Once the existence of minimizers and the main estimates are established, we can pass to the limit as $\varepsilon \to 0$ and $\beta \to +\infty$ and prove our main result.

\begin{thm}\label{Theorem:MainResult}
Let $\UU_0$ denote either $\UU_{\vv_0,\mathbf{g}_0}$ or $\UU_{\vv_0}$. Let $\vv_{\eps,	\beta}$ be a minimizer of $\FF_{\eps,\beta}$ in $\UU_0$.
Then, for every fixed $\eps \in (0,1)$, there exist a subsequence $\beta_n \to +\infty$ and a function $\vv_{\eps} \in \UU_{seg} \cap L^{\infty}(\Omega\times(0,\infty))^k$ such that, as $n \to \infty$,

\begin{equation}\label{eq:FinalLimit1}
\vv_{\eps,\beta_n} \to \vv_\eps \quad\text{strongly in $\UU$,
 and pointwise a.e. in $\Omega \times \R^+$.}
\end{equation}
Every such function $\vv_\eps$ is an  absolute minimizer
of the functional $\FF_\eps$ defined in \eqref{defFeps} on $\UU_{seg}$, and  satisfies the same
estimates as $\vv$ in
\eqref{uno} (with $\beta =0$), \eqref{due} and \eqref{tre}.

Furthermore, there exist a sequence $\eps_n \to 0$ and a function $\ww \in \UU_{seg} \cap L^{\infty}(\Omega\times(0,\infty))^k$ such that, as $n \to \infty$,
\begin{equation}\label{eq:FinalLimit2}
\vv_{\eps_n} \rightharpoonup \ww \quad\text{weakly in $\UU$, and pointwise a.e. in
  $\Omega \times \R^+$.}
\end{equation}
Finally, $\mathbf{w} \in \mathcal{S}$, the class of functions defined in \eqref{eq:Sclass}.
\end{thm}

\begin{remark} Under additional assumptions on the functions $F_i$ (for instance, if each $F_i$ is concave), it can be proved that there is no need to pass to  subsequences in \eqref{eq:FinalLimit1} and \eqref{eq:FinalLimit2} (see Remark \ref{Rem:MinFeps} and Remark \ref{Rem:UniqLimEps}). In other words, the limit $\ww \in \mathcal{S}$ is obtained as

\[
\ww = \lim_{\eps \to 0^+} \Big(\lim_{\beta \to +\infty} \vv_{\eps,\beta}\Big),
\]
where the double limit is  in the weak topology of $\UU$.
\end{remark}

Finally,  minimization of $\FF_{\eps,\beta}$ on $\UU_{\mathbf{g}_0}$ is treated in Section \ref{SectionDiscussions}, where we show (Theorem \ref{elliptic}) that we recover the results of the elliptic setting.

\section{Existence of minimizers}\label{SectionExistenceofMinimizers}
For every $\eps,\beta > 0$, it is convenient to introduce the rescaled functional

\begin{equation}\label{dgrv}
\JJ_{\eps,\beta}(\uu) := \int_0^\infty\!\!\int_{\Omega} e^{-t} \Big\{ |\de \uu|^2 + \eps \left(|\nabla \uu|^2 - 2F(\uu) + \frac{\beta}{2} \langle \uu^2, A \uu^2 \rangle\right)  \Big\} \,dxdt,
\end{equation}
which is equivalent to $\FF_{\eps,\beta}$, as defined in \eqref{eq:DeGiorgiFunctional}, in the sense that

\[
\FF_{\eps,\beta}(\vv) = \frac{1}{\eps} \JJ_{\eps,\beta}(\uu)
\]
whenever $\uu$ and $ \vv$ are related by the change of variable $\uu(x,t) = \vv(x,\eps t)$.
Since the convex sets $\UU_{\vv_0,\mathbf{g}_0}$, $\UU_{\vv_0}$ and $\UU_{\mathbf{g}_0}$
are invariant under this time scaling,
the
minimization of $\FF_{\eps,\beta}$ on each of these sets is equivalent to the minimization of $\JJ_{\eps,\beta}$ on the same set.
Working with $\JJ_{\eps,\beta}$, however, will turn out to be convenient
 in order to simplify the exposition, especially for what concerns
 the a priori estimates on minimizers. Now we state and prove the following result, which is easily seen to be equivalent to Theorem \ref{minimizers}.

\begin{thm}
\label{Theorem:ExistenceMinimizers}
Let $\UU_0$ denote one of $\UU_{\vv_0,\mathbf{g}_0}$, $\UU_{\vv_0}$ or $\UU_{\mathbf{g}_0}$.
For every $\eps \in(0,1)$ and $\beta > 0$, the functional $\JJ_{\eps,\beta}$ has an absolute minimizer  in $\UU_0$. Moreover, every minimizer
 $\uu=(u_1,\ldots,u_k)$ satisfies $0 \le u_i \le 1$ a.e. in $\Omega \times \R^+$, for every $i=1,\dots,k$.
\end{thm}

\begin{proof} First note that $\UU_0\neq\emptyset$ by \eqref{nonvuoto}
and that  $\JJ_{\eps,\beta}(\widetilde\vv)<+\infty$, so that the minimization is nontrivial.
Thanks to the assumptions \eqref{eq:AssumptionsPotential} on  $F$, the quantity
\begin{equation}
\label{F0}
M:= 2\sum_{i=1}^k \max_{v \in \R}\;  F_i(v)
\end{equation}
is finite and therefore, since the entries of the matrix $A$ are nonnegative,
for every $\uu \in \UU_0$ we have

\begin{equation}
\label{coerc}
\JJ\ebeta(\uu) \ge    \int_0^{\infty}  e^{-t}\int_{\Omega}
\left\{|\de \uu|^2 + \eps |\nabla \uu|^2\right\}\,  dxdt - \eps M|\Omega|,
\end{equation}
which in particular shows that $\JJ\ebeta$ is bounded from below on $\UU_0$.

Now let $\{\uu_n\} \subset \UU_0$ be a minimizing sequence, i.e.
\[
\JJ_{\eps,\beta}(\uu_n)\to \inf_{\uu \in \UU_0}\JJ_{\eps,\beta}(\uu).
\]
Due to the monotonicity assumptions
\eqref{eq:AssumptionsPotential} on $F_i$,
by replacing $\uu_n$ with
$\min\{1,\max\{\uu_n,0\}\}$ (where the truncations are meant componentwise)
we may assume that the minimizing sequence $\uu_n$ satisfies
\begin{equation}
\label{un}
0\le \uu_n \le 1
\end{equation}
(note that, after these truncations, we still have $\uu_n\in \UU_0$
by virtue of
\eqref{bound}).

We then see by \eqref{coerc} that, for every $T>0$,
there is a constant $C_T$ such that
\begin{equation}
\label{stimader}
\int_0^T  \int_{\Omega}
\left\{|\de \uu_n|^2 + |\nabla \uu_n|^2\right\}\,  dxdt\leq C_T.
\end{equation}
This inequality, combined with the $L^\infty$ bound \eqref{un},
reveals
that
$\{\uu_n\}$ is bounded in $\UU$,
and hence there exist a subsequence (still denoted $\uu_n$) and a function
$\uu \in \UU$ such that

\[
\text{$\uu_n \rightharpoonup \uu\,\,$ weakly in $\,\,\UU,\quad$ and
$\quad \uu_n \to \uu\,\,$  a.e. in $\,\Omega \times \R^+$.}
\]
Since  $\UU_0$ is closed and convex in $\UU$ and $\uu_n\in\UU_0$, we also have
$\uu\in\UU_0$.
Moreover  $0 \le \uu \le 1$ by \eqref{un} and,    by dominated convergence,
\begin{align*}
 \int_0^\infty \!\!\!\int_\Omega e^{-t} F(\uu_n)\,dxdt &\to \int_0^\infty
 \!\!\! \int_\Omega e^{-t} F(\uu)\,dxdt, \\
\int_0^\infty \!\!\!\int_\Omega e^{-t} \,\langle\uu_n^2, A\uu_n^2 \rangle\, dxdt
&\to \int_0^\infty\!\!\!  \int_\Omega e^{-t}\,\langle\uu^2, A\uu^2 \rangle \,dxdt.
\end{align*}
Since the terms involving derivatives in \eqref{dgrv} are
weakly lower semicontinuous, we deduce that
\[
\JJ\ebeta(\uu) \le \liminf_n \JJ\ebeta(\uu_n) = \inf_{\vv \in \UU_0} \JJ\ebeta(\vv),
\]
i.e. $\uu$ is an absolute minimizer of $\JJ$ on $\UU_0$. Finally, any minimizer $\uu$
must satisfy $0\leq \uu\leq 1$ almost everywhere, otherwise letting
$\vv=\min\{1,\max\{\uu,0\}\}$ we would have
$\JJ_{\eps,\beta}(\vv)<\JJ_{\eps,\beta}(\uu)$ (the strict inequality coming from the terms involving
derivatives).
\end{proof}

\section{Uniform estimates}\label{SectionEstimates}

This section is devoted to the proof of the main estimates on the  minimizers of
$\FF_{\eps,\beta}$.
Denoting by $\UU_0$ one of the convex sets
 $\UU_{\vv_0,\mathbf{g}_0}$ or $\UU_{\vv_0}$,
the main goal of this section is to prove Theorem \ref{ape}. The proof will be carried out in a series of lemmas, working
with the rescaled functional $\JJ_{\eps,\beta}$
defined in \eqref{dgrv}. Finally,
we will transfer the estimates to the minimizers of $\FF\ebeta$ by scaling the time.

\begin{lemma}\label{Lemma:LevelEstimate} (Level estimate)

Let $\uu$ be a minimizer of $\JJ\ebeta$ in $\UU_0$. Then

\[
|\JJ\ebeta(\uu)| \leq C \eps,
\]
for some constant $C$ independent of $\eps$ and $\beta$.
\end{lemma}

\begin{proof} Recalling \eqref{F0} and procceding as for \eqref{coerc},
neglecting the positive terms  we obtain
\[
\JJ\ebeta(\uu) \ge -  M |\Omega| \eps.
\]
On the other hand, the competitor function $\widetilde\vv$
defined in \eqref{nonvuoto}
satisfies  both $\de \widetilde\vv \equiv 0$ and
$\langle \widetilde\vv^2,A\widetilde\vv^2 \rangle \equiv 0$ thanks to the segregation assumption on the initial data \eqref{eq:SegregationCondition}. Consequently,

\[
\JJ\ebeta(\uu) \leq \JJ\ebeta(\widetilde \vv)  = \eps \int_0^{\infty} e^{-t}\!\!
\int_{\Omega} \left\{|\nabla \widetilde\vv|^2 - 2F(\widetilde\vv)\right\}\,dxdt = \eps
\int_{\Omega} \left\{|\nabla \vv_0|^2 - 2 F(\vv_0)\right\}\,dx  = C\eps,
\]
and the claim follows.
\end{proof}

We are now ready to prove the main technical result of the present work. To proceed further we must introduce some  notation that will be used in the statement the next lemma and in its proof.

If $\uu$ is a minimizer of $\JJ\ebeta$ in $\UU_0$, we  define
the time-dependent quantities

\begin{equation}\label{defRI}
R(t) = \eps \int_{\Omega} \left\{ |\nabla \uu|^2 - 2 F(\uu)
+ \frac{\beta}{2} \langle \uu^2, A \uu^2 \rangle
 \right\} \,dx,\quad
I(t) =   \int_{\Omega} |\de\uu|^2 dx,
\end{equation}
which enable us to write the functional in \eqref{dgrv} as
\begin{equation}
  \label{rewriteJJ}
  \JJ\ebeta(\uu)=\int_0^\infty e^{-t}\bigl\{ I(t)+R(t)\bigr\}\,dt.
\end{equation}
We also define the ``energy function''
\begin{equation}\label{defE}
E(t):=e^t\int_t^{\infty} e^{-\tau} \bigl\{ I(\tau)+R(\tau)\bigr\} \,d\tau, \quad t\geq 0,
\end{equation}
and observe that $E(0) =   \JJ\ebeta(\uu)$ and
\begin{equation}\label{eprimozero}
E'(t)=-I(t)-R(t)+E(t),\quad\text{for a.e. $t>0$}.
\end{equation}
Adapting the ideas of the proof of Proposition 1.3 of \cite{SerraTilli2012:art},
we will now compute this derivative in a different way. From this key computation, the estimates stated in Theorem \ref{ape} will follow quite easily.

\begin{lemma}\label{Lemma:ENERGYDERIVATIVE}
Let $\uu$ be a minimizer of $\JJ\ebeta$ in $\UU_0$. Then
\begin{equation}\label{Eprimo}
E'(t) = - 2I(t) \quad \text{ for a.e. $t\geq 0$.}
\end{equation}
\end{lemma}

\begin{proof}
Given  an arbitrary test function $\eta \in C_0^{\infty}(\mathbb{R}^+)$, we define

\begin{equation}
\label{zeta}
\zeta(t) := \int_0^t\eta(\tau)d\tau \qquad\text{and}\qquad  \varphi(t) := t - \delta \zeta(t),\qquad
t\geq 0,
\end{equation}
where $\delta$ is a small parameter.
Clearly, $\varphi(0)= 0$ and $\varphi'(t) >0$ if $|\delta|$ is small enough, so that
$\varphi$ is a smooth diffeomorphism of $[0,+\infty)$.  Moreover, from the definition of $\varphi$ we see that its inverse $\psi := \varphi^{-1}$ satisfies
\begin{equation}
\label{eq:EQUATIONFORPSI}
\psi(\tau) = \tau + \delta \zeta(\psi(\tau)), \quad \tau \geq 0.
\end{equation}
Now we use this diffeomorphism to perform an inner
variation of $\uu$,
by means of the competitor

\[
\ww(x,t) := \uu(x,\varphi(t))
\]
(note that $\ww\in\UU_0$, since $\ww(x,0)=\uu(x,0)$ while the boundary condition
along $\partial\Omega\times\R^+$, if prescribed as in \eqref{eq:DFSpace} or in \eqref{eq:EFSpace}, is maintained because $\mathbf{g}_0$ is time-independent). The dependence of $w$ on $\delta$, though crucial, is omitted
for notational simplicity (observe, however, that $\ww\equiv \uu$ when $\delta=0$).
Recalling \eqref{rewriteJJ}, we have

\[
\JJ\ebeta(\ww) = \int_0^\infty\!\! e^{-t}\left\{  \varphi'(t)^2
I\bigl(\varphi(t)\bigr)+ R\bigl(\varphi(t)\bigr)\right\}
\,dt
\]
and, by the change of variable  $t=\psi(\tau)$, since $\varphi(t)=\tau$ we obtain

\begin{align*}
\JJ\ebeta(\ww)  = \int_0^\infty\!\!
\psi'(\tau)e^{-\psi(\tau)} \Big\{ \varphi'(\psi(\tau))^2
I(\tau)+R(\tau)
\Big\} \,d\tau.
\end{align*}
Using the bound $e^{-\psi(\tau)} \leq  e^{\delta\|\zeta\|_{\infty}}e^{-\tau}$,
which follows
from \eqref{eq:EQUATIONFORPSI},
and the fact that $\psi'\in L^\infty(\R^+)$, we see that
$\JJ\ebeta(\ww)$ is finite.
In fact, since $\uu$ is a minimizer
and $\ww \equiv \uu$ when $\delta = 0$, we must have

\begin{equation}\label{var}
\frac{d}{d\delta}\JJ\ebeta(\ww) \Big|_{\delta = 0} = 0
\end{equation}
(the existence of the derivative will be clear, by dominated convergence,
in the light of the computations that follow).
Following the proof of \cite[Proposition 3.1]{SerraTilli2012:art},
(see also \cite[Lemma 4.5]{BogeleinEtAl2014:art}), differentiating \eqref{eq:EQUATIONFORPSI} with respect to  $\delta$ yields

\[
\frac{\partial}{\partial\delta} \psi(\tau) = \zeta(\psi(\tau)) + \delta\zeta'(\psi(\tau)) \frac{\partial}{\partial\delta}\psi(\tau) \]
so that
\[
\frac{\partial}{\partial\delta} \psi(\tau) \Big|_{\delta = 0} = \zeta(\tau),
\]
where we have used that $\psi(\tau)|_{\delta = 0} = \tau$. Similarly,

\[
\frac{\partial}{\partial\delta} \psi'(\tau) = \zeta'(\psi(\tau))\psi'(\tau) + \delta\Big(\zeta''(\psi(\tau))\psi'(\tau)\frac{\partial}{\partial\delta}\psi(\tau) + \zeta'(\psi(\tau)) \frac{\partial}{\partial\delta}\psi'(\tau)\Big),
\]
from which we obtain

\[
\frac{\partial}{\partial\delta} \psi'(\tau)\Big|_{\delta = 0} = \zeta'(\tau),
\]
and thus also

\[
\frac{\partial}{\partial\delta}\Big(\psi'(\tau)e^{-\psi(\tau)}\Big)\Big|_{\delta = 0} = \zeta'(\tau)e^{-\tau} - \zeta(\tau)e^{-\tau}.
\]
Finally, we compute

\[
\frac{\partial}{\partial\delta} |\varphi'(\psi(\tau))|^2 = 2 \varphi'(\psi(\tau)) \frac{\partial}{\partial\delta} \varphi'(\psi(\tau))
\]
to conclude that
\[
\frac{\partial}{\partial\delta} |\varphi'(\psi(\tau))|^2 \Big|_{\delta = 0} = -2\zeta'(\tau),
\]
where we have used the expression of the derivative $\varphi'(\psi(\tau)) =  1 -\delta \zeta'(\psi(\tau))$. Consequently, we can make \eqref{var} more explicit, namely

\begin{align}
\label{start}
0 & = \frac{d}{d\delta}\JJ\ebeta(\ww) \Big|_{\delta = 0}
 = \int_0^{\infty} \Big(\zeta(\tau)e^{-\tau}  \Big)'
 \Big\{  I(\tau)+R(\tau)  \Big\}   \,d\tau
 -2 \int_0^{\infty}  e^{-\tau}  \zeta'(\tau)I(\tau)\,d\tau.
\end{align}
Now, recalling \eqref{zeta}, a standard smoothing argument shows that
in this identity one can choose
\[
\zeta(\tau):=\begin{cases}
  0 & \text{if $\tau \leq t$,}\\
  e^t(\tau-t)/\lambda & \text{if $t<\tau<t+\lambda$,}\\
  e^t & \text{if $\tau \geq t+\lambda$,}
\end{cases}
\]
where $t>0$ is fixed while $\lambda>0$ is a small parameter. Letting $\lambda\to 0^+$,
one finds for a.e. $t>0$
\[
0=I(t)+   R(t)
-e^t\int_t^{\infty} e^{-\tau}
 \big\{ I(\tau)+  R(\tau) \big\} \,d\tau
-2I(t),
\]
i.e. $I+R-E=2I$. Recalling \eqref{eprimozero}, \eqref{Eprimo} is established.
\end{proof}

Now we state and prove two important corollaries of the above lemma.

\begin{cor}
Let $\uu $ be a minimizer of $\JJ\ebeta$ in $\UU_0$. Then
\begin{align}
|E(t)| &\leq C \eps \quad \forall t \geq 0, \label{enbound}\\
\int_0^{\infty} I(\tau) \,d\tau & \leq C \eps,\label{derbound}
\end{align}
for some constant $C$ independent of $\eps$ and $\beta$.
\end{cor}

\begin{proof} Since by \eqref{Eprimo} the function $E(t)$ is nonincreasing, for every
$t\geq 0$
\[
E(t) \le E(0) = \JJ\ebeta(\uu) \le C\eps
\]
by \eqref{defE}, \eqref{rewriteJJ} and Lemma \eqref{Lemma:LevelEstimate}. On the other hand,
recalling \eqref{F0}, by \eqref{defE} and \eqref{defRI}
\[
E(t)   \geq - e^t \int_t^{\infty} e^{-\tau} \int_{\Omega} 2\eps\sum_{i=1}^k F_i(u_i) \,dx d\tau  \geq - \eps M |\Omega| e^t \int_t^{\infty} e^{-\tau} d\tau = -C \eps.
\]
Next, integrating \eqref{Eprimo} yields

\[
2\int_0^t I(\tau) \, d\tau = E(0) - E(t) \leq E(0) \leq  C\eps
\]
thanks to \eqref{enbound}, and
\eqref{derbound} is obtained letting $t \to + \infty$.
\end{proof}

\begin{cor}\label{mainest}
If $\uu $ is a minimizer of $\JJ\ebeta$ in $\UU_0$, then

\begin{equation}\label{eq:stima1}
\int_t^{t+1}\!\int_{\Omega} \left\{
|\nabla \uu|^2 + \beta\langle \uu^2, A \uu^2 \rangle\right\}\,dxd\tau \leq C
\quad \forall t \geq 0,
\end{equation}
for some constant $C$ independent of $\eps$ and $\beta$.
\end{cor}

\begin{proof} Recalling \eqref{defRI}  and \eqref{F0}, we have for a.e. $\tau\geq 0$
\begin{align*}
\eps &\int_{\Omega} \left\{
|\nabla \uu(x,\tau)|^2 + \frac \beta 2\langle \uu(x,\tau)^2, A \uu(x,\tau)^2 \rangle\right\}\,dx \\
&=R(\tau)+2\eps
\int_{\Omega} F\bigl(\uu(x,\tau)\bigr) \,dx
\leq
R(\tau)+\eps M |\Omega|.
\end{align*}
On the other hand,
$R=E+I$ by \eqref{eprimozero} and \eqref{Eprimo}, so that
\[
\eps \int_{\Omega} \left\{
|\nabla \uu(x,\tau)|^2 + \frac \beta 2\langle \uu(x,\tau)^2, A \uu(x,\tau)^2 \rangle\right\}\,dx
\leq E(\tau)+I(\tau)+\eps M |\Omega|
\leq I(\tau)+C_1\eps
\]
having used \eqref{enbound}. By integrating this inequality over $(t,t+1)$
and dividing by $\eps$, one obtains
\eqref{eq:stima1} using \eqref{derbound}.
\end{proof}

\begin{proof}[Proof of Theorem \ref{ape}.]
If  $\vv $ is a minimizer of $\FF\ebeta$ over $\UU_0$ then, as already observed, the  function
$\uu(x,\tau)=\vv(x,\eps \tau)$ minimizes  $\JJ\ebeta$ over the same $\UU_0$.
Then estimate \eqref{due} is immediate since $0 \le \uu \le 1$ by Theorem \ref{Theorem:ExistenceMinimizers}, while \eqref{tre}, recalling the definition of $I$ in
 \eqref{defRI}, is obtained from the change of variable
 $\tau=\eps t$ in \eqref{derbound}. Similarly, the change of variable $\tau = s/\eps$
 in \eqref{eq:stima1} yields

\[
\int_{\eps t}^{\eps t + \eps} \int_{\Omega}  \left\{
|\nabla \vv|^2 + \beta\langle \vv^2, A \vv^2 \rangle\right\} \,dxds \leq C\eps \quad \text{ for every } t \geq 0.
\]
Choosing $t=\tau/\eps+j$ and summing the
corresponding inequality with $j=0,1,\ldots, \kappa-1$ gives

\[
\int_{\tau}^{\tau + \kappa\eps} \int_{\Omega}
\left\{
|\nabla \vv|^2 + \beta\langle \vv^2, A \vv^2 \rangle\right\}
\,dxds \leq C\kappa\eps \quad \forall \tau \geq 0
\]
and, given $T\geq\eps$, \eqref{uno} follows choosing $\kappa=\lceil T/\eps\rceil$.
\end{proof}

\section{The double limit}\label{double}

This section is devoted to the proof of Theorem \ref{Theorem:MainResult} and we begin by proving two lemmas.
We recall that  $\UU_0$ stands for  $\UU_{\vv_0,\mathbf{g}_0}$ or $\UU_{\vv_0}$.

\begin{lemma}
\label{limbeta}
For every $\eps\in (0,1)$ and every $\beta>0$, let
 $\vv_{\eps,\beta}\in \UU_0$ be a minimizer of $\FF_{\eps,\beta}$
 on $\UU_0$. Then, for every fixed $\eps$, there exist a subsequence $\beta_n \to +\infty$ and a function $\vv_{\eps} \in \UU_{seg} \cap L^{\infty}(\Omega\times(0,\infty))^k$ such that, as $n \to \infty$,

\begin{equation}\label{conv1Strong}
\vv_{\eps,\beta_n} \to \vv_\eps \quad\text{strongly in $\UU$, and pointwise a.e. in $\Omega \times \R^+$},
\end{equation}
and, for every $T > 0$,

\begin{equation}\label{eq:limitCompTerm}
\lim_n \, \beta_n \int_0^T\!\!\int_{\Omega} \langle \vv_{\eps,\beta_n}^2, A \vv_{\eps,\beta_n}^2 \rangle \,dxdt = 0.
\end{equation}
Every such function $\vv_\eps$ is an  absolute minimizer
of the functional $\FF_\eps$, defined in \eqref{defFeps}, on $\UU_{seg}$, and  satisfies the same
estimates as $\vv$ in
\eqref{uno} (with $\beta =0$), \eqref{due} and \eqref{tre}.
 \end{lemma}

\begin{proof} Estimates \eqref{uno}, \eqref{due} and \eqref{tre} of Theorem \ref{ape}, applied with $\vv=\vv\ebeta$ and $\tau=0$, reveal that
the family $\{\vv_{\eps,\beta} \}$ is equibounded in $H^1(\Omega_T)^k$ for every $T>0$,
as well as in $L^\infty(\Omega \times \R^+)^k$.
Therefore, keeping $\eps>0$ fixed, there exist a sequence $\{\vv_{\eps,\beta_n}\}$, with $\beta_n \to \infty$, and a function $\vv_\eps \in \UU$ satisfying

\begin{equation}\label{conv1}
\vv_{\eps,\beta_n} \rightharpoonup \vv_\eps \quad\text{weakly in $\UU$,
 and pointwise a.e. in $\Omega \times \R^+$.}
\end{equation}
Moreover, since $\UU_0$ is closed and convex, also $\vv_\eps\in\UU_0$,
and by pointwise convergence $\vv_\eps \in L^\infty(\Omega \times \R^+)^k$. Furthermore,
using first Fatou's lemma, and then \eqref{uno} with $\tau=0$, we see that
\[
\int_0^T\!\!\int_\Omega \langle A\vv_\eps^2, \vv_\eps^2\rangle \dx\dt\leq
\liminf_{n\to\infty}
\int_0^T\!\!\int_\Omega \langle A\vv_{\eps,\beta_n}^2, \vv_{\eps,\beta_n}^2\rangle \,\dx\dt
\leq \liminf_{n\to\infty} \frac{CT}{\beta_n}=0
\]
for every $T>0$, which shows that $\vv_\eps \in \UU_{seg}$.

To show that the function $\vv_\eps$  satisfies the required  estimates,
it suffices to rewrite
\eqref{uno} (with $\beta =0$), \eqref{due} and \eqref{tre}
with $\vv=\vv_{\eps,\beta_n}$, and pass to the limit: weak lower semicontinuity preserves all the inequalities, which are then inherited by $\vv_\eps$.

Now we show that each $\vv_\eps$ minimizes $\FF_\eps$ on $\UU_{seg}$. To this aim, we first notice that by weak lower semicontinuity it follows
\begin{equation}\label{eq:LSCFinal}
\FF_\eps(\vv_\eps) \le \liminf_n \FF_{\eps}(\vv_{\eps,\beta_n})
\le \liminf_n \FF_{\eps,\beta_n}(\vv_{\eps,\beta_n}).
\end{equation}
Then we take an arbitrary $\ww \in \UU_{seg}$ and observe that $\FF_{\eps,\beta_n}(\ww)=\FF_{\eps}(\ww)$ since $\ww$, being segregated, pays no penalization. Since moreover $\ww \in \UU_0$ and $\vv_{\eps,\beta_n}$ minimizes $\FF_{\eps,\beta_n}$ on $\UU_0$, we have

\[
\liminf_n \FF_{\eps,\beta_n}(\vv_{\eps,\beta_n}) \le \liminf_n \FF_{\eps,\beta_n}(\ww) = \liminf_n \FF_\eps (\ww) = \FF_\eps(\ww),
\]
and so $\FF_\eps(\vv_\eps) \leq \FF_\eps(\ww)$, for all $\ww \in \UU_{seg}$, i.e. $\vv_\eps$ is an absolute minimizer.

We are left to show that $\vv_{\eps,\beta_n} \to \vv_\eps$ strongly in $\UU$ and that \eqref{eq:limitCompTerm} holds. To do so, it is enough to notice that since $\vv_\eps \in \UU_{seg}$, we have the bound

\[
\FF_{\eps,\beta_n}(\vv_{\eps,\beta_n}) \leq \FF_{\eps,\beta_n}(\vv_\eps) = \FF_\eps(\vv_\eps) ,
\]
and thus, thanks to \eqref{eq:LSCFinal}, we  obtain

\begin{equation}\label{eq:LimitFunctionalbetaMin}
\FF_\eps(\vv_\eps) \le \liminf_n \FF_{\eps,\beta_n}(\vv_{\eps,\beta_n}) \le 	\limsup_n \FF_{\eps,\beta_n}(\vv_{\eps,\beta_n}) \le \FF_\eps(\vv_\eps),
\end{equation}
showing that $\FF_{\eps,\beta_n}(\vv_{\eps,\beta_n}) \to \FF_\eps(\vv_\eps)$ as $n\to \infty$.
Consequently, and by \eqref{eq:LSCFinal}, we deduce that

\begin{align*}
\limsup_n \frac{1}{2} \beta_n& \int_0^\infty\!\!\int_{\Omega} \frac{e^{-t/\eps}}{\eps} \langle \vv_{\eps,\beta_n}^2, A \vv_{\eps,\beta_n}^2 \rangle \,dxdt = \limsup_n \left(\FF_{\eps,\beta_n}(\vv_{\eps,\beta_n}) - \FF_{\eps}(\vv_{\eps,\beta_n})\right)
\\
 & \le \lim_n  \FF_{\eps,\beta_n}(\vv_{\eps,\beta_n}) -\liminf_n  \FF_{\eps}(\vv_{\eps,\beta_n}) \le \FF_\eps(\vv_\eps) - \FF_\eps(\vv_\eps) =  0
\end{align*}
and \eqref{eq:limitCompTerm} is proved. Now the strong convergence is a direct consequence of \eqref{eq:limitCompTerm} and \eqref{eq:LimitFunctionalbetaMin}.
\end{proof}

\begin{remark}\label{Rem:MinFeps} If one assumes that the all the functions $F_i$ are concave, then one can prove that the functional $\FF_\eps$ has a unique minimizer on $\UU_{seg}$. It suffices to repeat step by step the proof of Theorem 4.1 in \cite{ConTerVer2005:art} (the fact that \cite{ConTerVer2005:art} deals with an elliptic problem does not pose any difficulty, since the extra term $|\de \vv|^2$ in
$F_\eps$ behaves exactly like $|\nabla \vv|^2$ in all the computations, and the weight is just a coefficient). In particular this shows that, under the concavity assumption, there is no need to pass to subsequences in the preceding lemma.
\end{remark}

\begin{lemma}
\label{disug}
If $\vv_\eps = (v_\eps^{1},\ldots,v_\eps^{k})$ is one of the functions provided by  Lemma \ref{limbeta},
then
\begin{equation}
\label{disugset1}
\int_0^\infty\!\!\int_{\Omega} \Big\{\eta \de v_\eps^i
 +\eps \de v_\eps^{i} \de \eta
 + \nabla v_\eps^{i} \cdot \nabla \eta - f_i(v_\eps^{i})\eta
\Big\}\, dxdt \leq 0,\quad \forall
\eta \in C_0^\infty(\Omega\times \R^+),\quad \eta\geq 0,
\end{equation}
for every $i=1,\ldots,k$, and
\begin{equation}
\label{disugset2}
\int_0^\infty\!\!\int_{\Omega} \Big\{\eta \de \widehat v_\eps^i
 +\eps \de \widehat v_\eps^i \de \eta
 + \nabla \widehat v_\eps^i \cdot \nabla \eta - \widehat f_i \eta
\Big\}\, dxdt \geq 0,\quad \forall
\eta \in C_0^\infty(\Omega\times \R^+),\quad \eta\geq 0,
\end{equation}
for every $i=1,\ldots,k$, where
$\widehat v_\eps^i=v_\eps^i-\sum_{j\neq i} v_\eps^j$ and
$\widehat f_i=f_i(v_\eps^i)-\sum_{j\neq i} f_j(v_\eps^j)$.
\end{lemma}

\begin{proof}
Recall that $\vv_\eps$ is obtained, as in \eqref{conv1},
as limit of functions $\vv_{\eps, \beta_n}$ that minimize
the corresponding functionals
$\FF_{\eps,\beta_n}$ on $\UU_0$.
Let $\vv_{\eps, \beta_n}=(w_n^1,\ldots,w_n^k)$ and focus, say, on the
first component $w_n^1$.
Since $\vv_{\eps, \beta_n}$ minimizes
$\FF_{\eps,\beta_n}$ on $\UU_0$,
the vanishing of the first variation of  $\FF_{\eps,\beta_n}$ entails,
in particular, that
\[
\int_0^\infty\!\!\int_{\Omega} \frac{e^{-t/\eps}}\eps \Big\{\eps \de w_n^1 \de \varphi + \nabla w_n^1 \cdot \nabla \varphi - f_1(w_n^1)\varphi
+ \beta w_n^1\varphi \sum_{j=1}^k a_{1j}(w_n^j)^2 \Big\}\, dxdt =0,
\]
for every test function
 $\varphi \in C_0^\infty(\Omega\times \R^+)$. If, in addition, $\varphi\geq 0$,
neglecting the last term (which is positive) we obtain
\begin{equation*}
\int_0^\infty\!\!\int_{\Omega} e^{-t/\eps} \Big\{\eps \de w_n^1 \de \varphi + \nabla w_n^1 \cdot \nabla \varphi - f_1(w_n^1)\varphi
\Big\}\, dxdt \leq 0,\quad \forall
\varphi \in C_0^\infty(\Omega\times \R^+),\quad \varphi\geq 0.
\end{equation*}
This is best exploited if
 we choose $\varphi(x,t) =  e^{t/\eps}\eta(x,t)$, which gives
\begin{equation*}
\int_0^\infty\!\!\int_{\Omega} \Big\{\eta \de w_n^1
 +\eps \de w_n^1 \de \eta
 + \nabla w_n^1 \cdot \nabla \eta - f_1(w_n^1)\eta
\Big\}\, dxdt \leq 0,\quad \forall
\eta \in C_0^\infty(\Omega\times \R^+),\quad \eta\geq 0.
\end{equation*}
On the other hand, according to \eqref{conv1}, the first component
$w_n^1$ converges (weakly in $H^1(\Omega\times (0,T))$ for every $T>0$)
to $v_\eps^1$, the first component of the vector function $\vv_\eps$.  Thus,
letting $n\to\infty$ in the last inequality, one obtains
\eqref{disugset1} when $i=1$
(when $i>1$, the proof is the same).

The proof of \eqref{disugset2} is longer but more direct: it relies only on the fact that
$\vv_\eps$ minimizes $\FF_\eps$ over $\UUseg$, by
 constructing special competitors in the spirit of \cite{ConTerVer2005:art}. So,
focusing for instance on the first component of $\vv_\eps$,
according to the statement of the lemma we define
\[
\widehat v_\eps^1 = v_\eps^1- \sum_{j=2}^k  v_\eps^j,
\]
and we consider an arbitrary test function $\psi \in C_0^\infty(\Omega\times \R^+)$
such that $\psi\geq 0$.
Then, for $\delta\geq 0$, we construct a competitor $\vv^\delta = (v_1^\delta,\ldots,v_k^\delta)$, by setting
\[
v^\delta_1=(\widehat v_\eps^1 +\delta \psi)^+,\qquad
v^\delta_j =
 ( v_\eps^j -\delta \psi)^+ \quad\forall  j \in\{2,\ldots,k\}.
 \]
It is easy to see that $v_i^{\delta} \cdot v_j^{\delta} = 0$ a.e.
in $\Omega\times(0,\infty)$ for every $i \not= j$, so that $\vv^\delta \in \UU_{seg}$
and therefore
\begin{equation}
\label{minimo}
\FF_{\varepsilon}(\vv^{\delta}) \geq \FF_{\eps}(\vv_\eps)\quad\forall\delta\geq 0.
\end{equation}
Observing that for a.e. $(x,t)$,
\[
\sum_{j=2}^k |\nabla v^\delta_j(x,t)|^2=\left|\nabla \,\bigl(\widehat v_\eps^1(x,t) +\delta\psi(x,t)\bigr)^-\right|^2,
\]
we see that
\[
|\nabla \vv^\delta|^2= |\nabla \widehat v^\delta_1|^2 + \sum_{j=2}^k  |\nabla \widehat v^\delta_j |^2
= |\nabla(\widehat v_\eps^1 +\delta \psi)^+|^2 +  |\nabla(\widehat v_\eps^1 +\delta \psi)^-|^2 =  |\nabla(\widehat v_\eps^1 +\delta \psi)|^2
\]
and, since moreover $|\nabla \widehat v_\eps^1|^2=|\nabla\vv_\eps|^2$, we find
\begin{equation}
\label{grad}
\int_0^\infty\!\!\!\int_\Omega \emt  |\nabla \vv^\delta |^2  \dx\dt =
\int_0^\infty\!\!\!\int_\Omega \emt  \left\{ |\nabla \vv_\eps |^2 +
2\delta \nabla \widehat v_\eps^1 \cdot \nabla \psi +\delta^2|\nabla\psi|^2\right\}\,dxdt.
\end{equation}
In exactly the same way, for time derivatives we have
\begin{equation}
\label{timeder}
\int_0^\infty\!\!\!\int_\Omega \emt  |\de \vv^\delta |^2  \dx\dt =
\int_0^\infty\!\!\!\int_\Omega \emt  \left\{ |\de \vv_\eps |^2 +
2\delta \de \widehat v_\eps^1 \de \psi +\delta^2|\de\psi|^2\right\}\,dxdt.
\end{equation}

Finally we notice that, by straightforward computations,
\begin{equation}
\label{fi}
\frac{\partial}{\partial \delta} F_1 (v_1^\delta)\Big|_{\delta=0^+} =
f_1(v_\eps^1)\psi,\qquad
\frac{\partial}{\partial \delta} F_j (v_j^\delta)\Big|_{\delta=0^+} =
-f_j(v_\eps^j)\psi,\quad j=2,\ldots,k.
\end{equation}
By \eqref{grad}, \eqref{timeder} and \eqref{fi}, letting
$\widehat f_1=f_1(v_\eps^1)-\sum_{j\neq 2} f_j(v_\eps^j)$ and
recalling \eqref{defFeps}, we can write
\[
\FF_\eps(\vv^\delta) = \FF_\eps(\vv_\eps) +2\delta
\int_0^\infty\int_\Omega \frac{\emt}\eps
\Big\{\eps \de \widehat v_\eps^1 \de \psi +\nabla \widehat v_\eps^1 \cdot \nabla \psi -
\widehat f_1\psi \Big\}\dx\dt + o(\delta)
\]
as $\delta \to 0^+$ and therefore, due to \eqref{minimo}, we must have
\[
\int_0^\infty\int_\Omega \emt \Big\{\eps \de \widehat v_\eps^1 \de \psi
+\nabla \widehat v_\eps^1 \cdot \nabla \psi -  \widehat f_1\psi \Big\}\dx\dt \ge 0.
\]
Finally, choosing as before $\varphi(x,t) =  e^{t/\eps}\eta(x,t)$
where $\eta\in C_0^\infty(\Omega\times \R^+)$ and $\eta\geq 0$,
one obtains
\eqref{disugset1} when $i=1$
(when $i>1$, the proof is the same).
\end{proof}

\begin{proof}[Proof of Theorem \ref{Theorem:MainResult}]
Let $\{\vv_{\eps,\beta}\}$ be the family of minimizers of $\FF_{\eps,\beta}$ found in Theorem \ref{minimizers}.
By Lemma \ref{limbeta}, for every fixed $\eps \in (0,1)$ there exist a subsequence $\beta_n \to +\infty$ and a function $\vv_{\eps} \in \UU_{seg} \cap L^{\infty}(\Omega\times(0,\infty))^k$ such that, as $n \to \infty$,

\[
\vv_{\eps,\beta_n} \to \vv_\eps \quad\text{strongly in $\UU$, and pointwise a.e. in $\Omega \times \R^+$}.
\]
Moreover, every such function $\vv_\eps$ is an  absolute minimizer
of the functional $\FF_\eps$, defined in \eqref{defFeps}, on $\UU_{seg}$, and  satisfies the same
estimates as $\vv$ in
\eqref{uno} (with $\beta =0$), \eqref{due} and \eqref{tre}.
Therefore, there exist a sequence $\eps_n \to 0$ and a function $\ww \in \UU_0$ such that, as $n \to \infty$,
\[
\vv_{\eps_n} \rightharpoonup \ww \quad\text{weakly in $\UU$, and pointwise a.e. in
  $\Omega \times \R^+$.}
\]
The pointwise convergence also shows that $\ww \in  \UU_{seg} \cap L^\infty(\Omega\times \R^+)$.
By Lemma \ref{disug}, each $\vv_\eps$ satisfies   \eqref{disugset1}
and \eqref{disugset2}. Hence,
to show that $\ww$ satisfies the inequalities  \eqref{eq:DiffIneq},
it suffices to let $\eps\to 0$ (along the sequence $\eps_n$) in \eqref{disugset1}
and  \eqref{disugset2}, for every fixed $i=1,\ldots,k$.
\end{proof}

\begin{remark}\label{Rem:UniqLimEps} By \cite[Theorem 1.2]{Wang2015:art}, under additional regularity assumptions on the boundary data $\vv_0$ and $\mathbf{g}_0$, the system \eqref{bc}-\eqref{eq:DiffIneq} has at most one weak solution. As a consequence, the weak limit of the family $\{\vv_\eps\}_{\eps > 0}$ exists and coincides with $\ww \in \mathcal{S}$.
\end{remark}

\section{Comments and open problems}\label{SectionDiscussions}

We end the paper with some comments and open problems.

\paragraph{The elliptic setting.} In this paragraph we briefly describe how the existence of segregated solutions in the elliptic case, treated in \cite{ConTerVer2002:art,ConTerVer2005:art} with variational techniques, can be recovered as a particular case of Theorem \ref{Theorem:MainResult}.

Consider the functional  $\mathcal{E}_{\beta}: H^1(\Omega)^k \to (-\infty,+\infty]$ defined by
\[
\mathcal{E}_{\beta}(\vv) :=  \int_{\Omega} \left\{ |\nabla \vv|^2 - 2F(\vv) +
\frac\beta 2 \langle \vv^2, A \vv^2 \rangle  \right\} dx.
\]
In the next Proposition we link the functional $\mathcal{E}_\beta$ with the functional $\FF_{\eps,\beta}$, while in Theorem \ref{elliptic}, we show how our analysis allows one to recover the convergence results in the elliptic framework (see for instance \cite[Section 5]{ConTerVer2005:art}, \cite[Section 8]{TavTer2012:art}).

\begin{prop}
For $\eps,\beta >0$, let $\vv(t,x)$ be a minimizer of $\FF\ebeta$ over $\UU_{\mathbf{g}_0}$.
Then $\vv$ is time-independent, and the function $x\mapsto \vv(t,x)$ solves
the problem
\begin{equation}
\label{minE}
\min\left\{ \EE_\beta(\ww)\,|\,\, \ww\in H^1(\Omega),\quad \text{$\ww=\mathbf{g}_0$
on $\partial\Omega$}\right\}.
\end{equation}
Conversely, if $\uu$ is a solution to this problem, then the time-independent function
$(t,x)\mapsto \uu(x)$ is a minimizer of $\FF\ebeta$ over $\UU_{\mathbf{g}_0}$.
\end{prop}

\begin{proof}
The link between the two functionals is given by the identity
\[
\eps\FF\ebeta(\vv)=\eps \int_0^\infty\!\!\!\int_\Omega
\emt |\de \vv|^2\,dxdt+
\int_0^\infty\emt\, \EE_\beta\bigl(\vv(t,\cdot)\bigr)\,dt\quad\forall \vv\in\UU,
\]
which follows immediately from \eqref{dgv}.
Now let $\uu$ be a solution to \eqref{minE} and let $\widehat \uu(x,t):=\uu(x)$, so that
$\widehat \uu\in \UU_{\mathbf{g}_0}$ and $\de\widehat\uu\equiv 0$. Then, for every $\vv\in\UU_{\mathbf{g}_0}$,
using the previous identity we obtain
\begin{equation}
\label{ch1}
\eps\FF\ebeta(\widehat\uu)=\int_0^\infty\emt\, \EE_\beta\bigl(\widehat\uu(t,\cdot)\bigr)\,dt
=\int_0^\infty\emt\, \EE_\beta(\uu)\,dt
\leq \int_0^\infty\emt\, \EE_\beta\bigl(\vv(t,\cdot)\bigr)\,dt
\leq \eps\FF\ebeta(\vv)
\end{equation}
(we have used the fact that $\vv(t,\cdot)=\mathbf{g}_0$ on $\partial\Omega$ for
a.e. $t$, so that $\EE_\beta(\uu)\leq \EE_\beta\bigl(\vv(t,\cdot)\bigr)$ by the minimality of $\uu$).
If, moreover, $\vv$ minimizes $\FF\ebeta$ over $\UU_{\mathbf{g}_0}$, then $\FF\ebeta(\vv)
\leq\FF\ebeta(\widehat\uu)$, so that the two inequalities in \eqref{ch1} are in fact equalitites:
the latter equality then entails that $\de \vv\equiv 0$, i.e. $\vv$ is time-independent,
while the former, namely $\EE_\beta(\uu)= \EE_\beta\bigl(\vv(t,\cdot)\bigr)$, means
that also the function $x\mapsto \vv(t,x)$ solves \eqref{minE}.

Finally, \eqref{ch1} also shows that if $\vv$ minimizes $\FF\ebeta$
over $\UU_{\mathbf{g}_0}$, then so does $\widehat\uu$.
\end{proof}

In view of the equivalence of the minimization of $\FF\ebeta$ over $\UU_{\mathbf{g}_0}$ and the minimization of
$ \EE_\beta$ as in \eqref{minE}, we obtain, following step by step the proof of Theorem \ref{Theorem:MainResult}, the following stationary, or ``elliptic'' version of that result. Of course, to carry out this program, we assume that the function $\mathbf{g}_0$ is the trace of a segregated function $\mathbf{v}_0 \in H^1(\Omega)$.

\begin{thm}
\label{elliptic}
Let $\vv_\beta$ be a solution of \eqref{minE}.
Then there exist a sequence $\beta_n \to +\infty$ and a function $\ww \in H^1(\Omega)^k \cap L^{\infty}(\Omega)^k$ such that, as $n \to \infty$,
\[
\vv_{\beta_n} \rightharpoonup \ww \quad\text{weakly in $H^1(\Omega)^k$, and pointwise a.e. in
$\Omega$.}
\]
The components $(w_1,\dots,w_k)$ of  $\mathbf{w}$ satisfy  $\,w_i \cdot w_j = 0 \text{ a.e. in } \Omega$ for all $i\ne j$ and the system of differential inequalities
\[
\begin{cases}
- \Delta w_i  \leq f_i(w_i),\quad  i = 1,\ldots,k  \\
- \Delta \Big( w_i - \sum\limits_{j\not=i}w_j \Big ) \geq f_i(w_i) - \sum\limits_{j\not=i} f_j(w_j),
\quad i = 1,\ldots,k.
\end{cases}
\]
\end{thm}

\paragraph{Segregated solutions to more general reaction-diffusion systems.} As we mentioned in the introduction, our minimization approach is quite flexible. We conjecture that similar techniques can be applied to show the existence of segregated solutions to a much wider class of reaction-diffusion systems with strong competition. In this direction, it seems natural to study the asymptotic limit ($\beta \to +\infty$, $\eps \to 0^+$) of minimizers of the functional

\[
\FF_{\eps,\beta}(\vv) := \int_0^\infty\!\!\int_{\Omega} \frac{e^{-t/\eps}}{\eps} \Big\{ \eps |\de \vv|^2 +  \frac{\beta}{2} \langle \vv^2, A \vv^2 \rangle  \Big\} \,dxdt + \int_0^\infty \frac{e^{-t/\eps}}{\eps} \,\mathcal{W}(\vv(\cdot,t)) \,dt,
\]
where $\mathcal{W}$ is a functional defined on some Banach space of functions, in the spirit of \cite{SerraTilli2016:art} (see also \cite{AkagiStefanelli2016:art,MielkeStefanelli2011:art}).

\paragraph{Sharp regularity of solutions and study of the free boundary.} A second interesting development is the study of the regularity of the family of minimizers $\{\vv_{\eps,\beta}\}$ of \eqref{eq:DeGiorgiFunctional}. In view of \cite{CafLin2009:art,CafSal2005:book,DanWanZha2011:art,DanWanZha2012:art,DanWanZha2012_1:art,Snelson2015:art} (see also \cite{CafLin2008:art,ConTerVer2005_1:art,NorTavTerVer2010:art,SoaveZilio2015:art,TavTer2012:art} for the elliptic framework), it is reasonable to conjecture that $\vv_{\eps,\beta}$ satisfies some uniform H\"older (or even Lipschitz) bounds and the free boundary of segregated solutions satisfies some regularity and stratification properties. However, the methods developed in \cite{DanWanZha2012:art,DanWanZha2012_1:art} do not apply directly since the Euler-Lagrange equations of $\vv_{\eps,\beta}$ are \emph{elliptic} for any $\eps > 0$, and \emph{parabolic} in the limit $\varepsilon \to 0^+$. It is natural to expect that suitable modifications of the monotonicity formulas proved in \cite[Theorem 2.1, Theorem 4.1]{DanWanZha2011:art} hold in this  setting and allow one to prove the regularity estimates mentioned above. This will be the object of further research.
\bigskip

\noindent{\bf Acknowledgements.} The authors wish to thank Prof. Susanna Terracini for inspiring and fruitful discussions concerning the results presented in this paper.

\end{document}